\documentclass{svproc}
\usepackage{url}

\usepackage{amsmath}
\usepackage{amsfonts}
\usepackage{amssymb}
\usepackage{mathrsfs}

\newcommand{\vf}{\varphi}

\newcommand{\X}{\mathscr{X}}
\newcommand{\e}{\mathbf{e}}

\begin{document}
\mainmatter              
\title{Minimal unit vector fields on oscillator groups}
\titlerunning{Minimal unit vector fields}  
%
\author{Alexander Yampolsky}
\authorrunning{Alexander Yampolsky} 
%
%
\institute{V.N. Karazin Kharkiv National University, Ukraine\\
\email{a.yampolsky@karazin.ua}}

\maketitle              

\begin{abstract}
In this paper we treat minimal left-invariant unit vector fields on oscillator group and their relations with the ones that defines a harmonic map.  Particularly, if all structure constants of the oscillator group are equal to each other, then all unit leftinvariant vector fields that define a harmonic map into the unit tangent bundle with Sasaki metric are minimal.

\end{abstract}

\section{Introduction and preliminaries}
Let $(M,g)$ be Riemannian manifold of dimension $n$. Denote by $U(u^1,\ldots, u^n)$  a local chart on $M$, and by $\partial_i =\partial/\partial u^i$ the natural coordinate frame over $U$.  At each $p\in U$,  there is a decomposition $ V = V ^i\partial_i$ for each $ V \in T_pM$. Then  $(u^1, \ldots,u^n;  V ^1,\ldots,  V ^n)$ forms a local coordinate chart $TU$ on the tangent bundle $TM$.

Denote by $ds^2$ the line element of $(M,g)$. Denote by  $\nabla$ the Livi-Chivita connection for $g$. The Sasaki metric line element $d\sigma$ on the tangent bundle $TM$ is defined by local Pithagorean theorem, namely,
 $$
 d\sigma^2=ds^2+|D V |_g^2,
 $$
 where $D V ^i=du^k\nabla_k  V ^i =d V ^i+\Gamma^i_{jk} V ^jdu^k$ is a covariant differential of $ V ^i$.  At this point the $( V ^1,\ldots, V ^n)$ are independent variables. Denote by $g_S$ the  Riemannian metric on $TM$  with the line element $d\sigma^2$. \emph{Riemannian manifold $(TM,g_S)$ is referred to as tangent bundle with Sasaki metric.}

 Let $ V (u)$ be a smooth local tangent vector field on $M$. In terms of natural local coordinates $(U(u^1, \ldots,u^n);  V ^1,\ldots,  V ^n)$ the field $ V $  defines a local embedding $ V :U\to (TM,g_S)$ as
 $$
 \left\{
 \begin{array}{l}
   u^i=u^i,\\
  V ^i= V ^i(u^1,\ldots, u^n).
 \end{array}
 \right.
 $$
 In this approach a smooth local vector field can be understood as local explicitly given submanifold/graph  $ V (U)\subset (TM, g_S)$ \emph{with induced intrinsic and extrinsic geometry}.

If $ V $ can be given globally, then $ V (M)\subset (TM,g_S)$ is globally given graph in $(TM, g_S)$ diffeomorphic to the base manifold. Remark that if one do not pose any restrictions on the vector field, then $ V (M)$ does not inherit intrinsic geometry of the base manifold even locally, and its extrinsic geometry might be of any kind. To see this, it is sufficient to consider the case $g_{ij}=\delta_{ij}$, when $(TM,g_S)$ is isometric to $E^{2n}$ with Euclidean metric, and to look on a graph $ V (E^n)\subset E^{2n}$.  In contrast, if $g( V , V )=1$, then $ V (M)$ gives rise to some kind of equidistant, relative to the base, submanifold in $T_1M\subset (TM, g_S)$, and the geometry of $ V (M)$ becomes more predictable both, in intrinsic or extrinsic senses.

If $g( V , V )=1$, then $ V (M)$ is globally/locally given $n$ dimensional submanifold in $T_1M$ of codimension $(n-1)$ with the pull-back metric
\begin{equation}\label{MetricVF}
g_ V (X,Y)=g(X,Y)+g(\nabla_X  V ,\nabla_Y  V ).
\end{equation}
One can see that this metric is a local additive deformation of the metric $g$ on $M$ in presence of a (unit) vector field $ V $. At this point \emph{one can assign } the geometry of submanifold $ V (M)$ to  the field $ V $. One can talk about \emph{{sectional curvature}},  \emph{{Ricci curvature}},  \emph{{scalar}} curvature and other properties from Riemannian geometry.

In this paper we will focus on the \emph{extrinsic} properties of the submanifold $ V (M)\subset (TM, g_S)$ connected with the second fundamental form of (local) imbedding $ V (M)\subset (T_1M, g_S)$.

\subsection{Second fundamental form  of $ V (M)\subset (T_1M, g_S)$}

Denote by $\X(M)$ a module of smooth vector fields om $M$.
At each point $p\in M$ the point-wise \emph{Nomizu operator} $A_ V :T_p M\to  V ^\perp \subset T_pM)$ is defined by
$$
A_ V  X=-\nabla_X V .
$$
The \emph{adjoint} Nomizu operator  $A_ V ^t$ is defined by $$ g(A_ V  X,Y)=g(X,A_ V ^tY).$$
If $ V $ is smooth, then both operators form $(1,1)$ smooth tensor fields on $M$.
The {\emph{tangent}} $ V _*:\X(M)\to T V (M)$ and the {\emph{normal}} $\tilde n:\X(M)\to T^\perp V (M)$
mappings are defined by \cite{Ym-2003}
\begin{equation}\label{Framing}
 V _*(X)=X^h-(A_ V  X)^{t}=X^h-(A_ V  X)^v,\quad
\tilde n(X)=(A_ V ^tX)^h+X^{t},
\end{equation}
where $X^{t}$ means the vertical lift of $X-g(X, V ) V $ and is called by \emph{tangentional lift} of the corresponding object.

 The  \emph{rough $ V $ - Hessian} and the \emph{$ V $-harmonicity tensor } are given by
$$
Hess_ V (X,Y)=\frac12\big((\nabla_X A_ V )Y+(\nabla_Y A_ V )X\big)
$$
$$
Hm_ V (X,Y)=\frac12\big(R( V , A_ V  X)Y+R( V , A_ V  Y)X\big),
$$
where $(\nabla_X A_ V )Y=\nabla_X(A_ V  Y)-A_ V (\nabla_X Y)$ and $R$ -- is the curvature tensor of the base manifold $(M,g)$.
The
$$
trace (Hess_ V )=\sum (\nabla_{e_i} A_ V )e_i=\bar \Delta  V
$$
is known as the  {\emph{rough (or Bohner) Laplacian}}. The unit vector field is said to be  \textbf{harmonic}  if
$$
\bar\Delta V =|A_ V |^2  V
$$
with respect to some orthonormal frame $\{e_1,\ldots, e_n\}.$
If in addition
$$
trace (Hm_ V )=\sum R( V ,A_ V  e_i)e_i =0,
$$
then $ V $ defines a \textbf{\emph{harmonic map}} $ V :M\to (T_1M,g_S)$ of the base manifold into its unit tangent bundle with the Sasaki metric. The \emph{ second fundamental form of  mapping} $ V :M\to (T_1M,g_S)$ is given by
$$
B_ V (X,Y)=\big((Hm_ V (X,Y)\big)^h+\big(Hess_ V (X,Y)\big)^t.
$$
It follows that the tension field of the mapping $ V :M\to (T_1M,g_S)$ is given by
$$
\tau_ V =(trace(Hm_ V )^h+(\bar\Delta  V )^t.
$$
The {  \emph{second fundamental form of the submanifold}} $( V (M),g_ V )\subset (T_1M,g_S)$ with respect to the normal vector field $\tilde n(Z)$ is of the
form \cite{Ym-2003}
\begin{equation}\label{eqn:6}
\tilde\Omega_{\tilde n( Z)}( V _* X, V _* Y) =g\big(Hess_ V (X,Y)+A_ V  Hm_ V (X,Y),Z^\perp\big),
\end{equation}
where $Z^\perp=Z-g(Z, V ) V $.

 As a consequence, the \emph{second fundamental form vanish} if and only if
\begin{equation}\label{TGF}
Hess_ V (X,Y)+A_ V  Hm_ V (X,Y)-g(A_ V  X,A_ V  Y) V =0.
\end{equation}
The equation \eqref{TGF} provides a condition on  { \emph{totally geodesic property of the unit vector field}} \cite{Ym-2005}.
In spite the fact that the system \eqref{TGF} is overdefinite system of differential equations, in some cases the system can be solved \cite{Ym-2003}, \cite{Ym-2004}, \cite{Ym-2005}, \cite{Ym-2007}.

\subsection{Minimal unit vector fields}

At first, minimality of a unit vector field was defined in global setting by H. Gluck and W. Ziller \cite{GZ} by using the volume variation of a submanifold $ V (M)\subset (T_1M, g_S)$ relative to variation of the field within the class of unit vector fields. O. Gil-Medrano and E. Llinares-Fuster \cite{GM-LF}  proved that in local setting minimality of unit vector field is equivalent to minimality of a submanifold $ V (M)\subset (T_1M, g_S)$ relative to classic volume variation of a submanifold. In other words, locally given unit vector field $ V $ is locally minimal, if the submanifold $ V (M)\subset (T_1M, g_S)$ has zero mean curvature vector, and minimality problem for the unit vector field reduces to trace of the shape operator for $ V (M)\subset (T_1M, g_S)$.

Probably, the easiest way to rich the shape operator is a singular frame of the Nomizu operator. Since $ V \in \ker A_ V ^t$, the orthonormal singular frames for the Nomizu operator $A_ V $  consist of
$$
\{e_1,\ldots, e_{n-1}, e_{n},\}\in \X(M), \quad  \{ f_1,\ldots, f_{n-1},f_n= V \}\in  V ^\perp
$$
such that
$$
 A_ V  e_\alpha=\sigma_\alpha f_\alpha, \ A_ V  e_n =0, \quad  A_ V ^t f_\alpha=\sigma_\alpha e_\alpha\quad (\alpha=1,\ldots, n-1),
$$
where $\sigma_1\geq \sigma_2\geq\ldots\geq \sigma_{n-1}\geq\sigma_n=0 $ are singular values for $A_ V $. With respect to the singular frames the tangent and normal framing for $ V (M)$ consists of \cite{Ym-2002}
$$
\tilde e_n =e_n^h, \quad \tilde e_\alpha =\frac{e_\alpha ^h+\sigma_\alpha f_\alpha ^v}{\sqrt{1+\sigma_\alpha^2}}, \qquad \tilde n_\alpha=\frac{-\sigma_\alpha e_\alpha^h+ f_\alpha^v}{\sqrt{1+\sigma_\alpha ^2}},  \quad (\alpha=1,\ldots, n-1)
$$

The mean curvature vector of local embedding $ V (M)\subset T_1M$ with respect to this framing is of the form \cite{Ym-2002}
\begin{equation}\label{H}
H_ V =\frac{1}{n}\sum_{\alpha=1}^{n-1}  g\left( \sum_{i=1}^{n}\frac{(\nabla_{e_i}A_ V )e_i+A_ V  R( V ,A_ V  e_i )e_i}{1+\sigma_i^2},\frac{f_\alpha}{\sqrt{1+\sigma_\alpha^2}}\right)\tilde n_\alpha
\end{equation}
The expression \eqref{H}  has nice geometrical meaning. Denote
$$
\check{H}_ V =\sum_{i=1}^{n}\frac{(\nabla_{e_i}A_ V )e_i+A_ V  R( V ,A_ V  e_i )e_i}{1+\sigma_i^2}
$$
It is easy to see that
$
g(\check{H}_ V , V )=\left(\sum_{i=1}^{n} \frac{\sigma_i^2}{1+\sigma_i^2}\right).
$
Hence, the minimality condition  can be given  in the form
\begin{equation}\label{Minimality}
\sum_{i=1}^{n}\frac{(\nabla_{e_i}A_ V )e_i+A_ V  R( V ,A_ V  e_i )e_i}{1+\sigma_i^2}=\left(\sum_{i=1}^{n} \frac{\sigma_i^2}{1+\sigma_i^2}\right) V
\end{equation}
The latter means that \emph{ $ V $ is minimal iff $\check{H}_ V $ is collinear to $ V $.}

\subsection{The Reeb vector field}

The \emph{almost contact metric structure} on a smooth differentiable manifold $(M^{2n+1},g)$  of dimension $2n + 1$ consists of a unit vector field $\xi$, $(1,1)$ tensor field   $\varphi$ and  1-form  $\eta$ such that
\begin{equation}\label{Basic}
\varphi^2=-I+\eta\otimes\xi,\quad \varphi \xi = 0, \quad \eta(\xi)=1, \quad \eta \circ \varphi = 0
\end{equation}
\begin{equation}\label{BasicMetric}
g(\vf X,\vf Y)=g(X,Y)-\eta(X)\eta(Y),\quad \eta(X)=g(X,\xi)
\end{equation}
for all vector fields on the manifold.
It easily follows that $\vf$ is skew symmetric
\begin{equation}\label{phi_prop}
g(\varphi X, Y)=-g(X, \varphi Y), \quad  \varphi^t=-\varphi
\end{equation}
and orthogonal being restricted on $\ker \eta=\xi^\perp$
\begin{equation}\label{orth_Op}
(\varphi^t\varphi)X =(\varphi\varphi^t)X=X
\end{equation}
for any $X\in \ker \eta$.

The unit vector field  $\xi$ is called  \emph{characteristic} or  the   \emph{Reeb} vector field of contact metric manifold. The almost contact metric structure is denoted by $(\vf,\xi,\eta,g)$.
An almost contact metric structure  $(\vf,\xi,\eta,g)$ on  is called a $(\alpha,\beta)$-\textit{trans Sasakian}  structure \cite{Ob}  if
\begin{equation}\label{Trans_Sasaki}
(\nabla_X \varphi)Y =\alpha(g(X,Y)\xi - \eta(Y)X)+ \beta (g(\varphi X,Y)\xi-\eta(Y)\varphi X)
\end{equation}
for some smooth functions $\alpha, \beta :M^{2n+1}\to R$. The Structure is called {\emph{Sasakian}} if $\alpha=1$, $\beta=0$ ; {\emph{Kenmotsu}} if  $\alpha=0,\beta=1$; {\emph{cosymplectic}} if  $\alpha=0,\beta=0$.

It worthwhile to mention that for $ dim(M)\geq 5$ the trans-Sasakian manifold is either $\alpha$-Sasakian with $\alpha=const$  or $\beta$-Kenmotsu \cite{Marrero}.
\begin{theorem}\label{ReebTG} \cite{Ym-2022} The Reeb vector field on connected $(\alpha,\beta)$ trans-Sasakian manifold $M$  gives rise to totally geodesic submanifold
 $\xi(M)\subset (T_1M,g_S)$ only in the following cases
\begin{itemize}
\item $\beta=0$, $\alpha=1$ and hence $M$ is Sasakian or $\alpha=0$ and hence $M$ is cosymplectic;
\item $\alpha=0$ and $\nabla\beta=\dfrac{\beta^2(\beta^2+1)}{1-\beta^2}\xi$. If $\beta=const$ or $M$ is compact, then $\beta=0$ and hence $M$ is cosymplectic.
\end{itemize}
\end{theorem}
\begin{corollary} (see also \cite{Ym-2003})
The \emph{Hopf unit vector field} on every odd-dimensional unit sphere $S^{2m+1}$ ($m\geq1$) is  \emph{totally geodesic}.
\end{corollary}

\begin{theorem}\label{ReebH} \cite{Ym-2022} The norm of the  mean curvature vector $H_\xi$ for  the Reeb vector field $\xi$ on $(\alpha,\beta)$ trans-Sasakian manifold of $\dim M=2n+1$ is of the form
\begin{itemize}
\item if $\dim(M)>3$, then
$$
\begin{array}{l}
|H_\xi|=\frac{(1+(2n-1)\alpha^2)}{(2n+1)(1+\alpha^2)^{3/2}}|\vf^2\nabla\alpha|=0 \quad (\alpha=const, \beta=0);
\\[1ex]
|H_\xi|=\frac{(1+(2n-1)\beta^2)}{(2n+1)(1+\beta^2)^{3/2}}|\vf\bar\Delta\xi| \quad (\alpha=0),
\\[1ex]
\mbox{where $\bar\Delta$ stands for rough Laplacian;}
\end{array}
$$
\item if $\dim(M)=3$,  then
$$
|H_\xi|=\frac{|\vf^2\nabla\alpha-\vf\nabla\beta|}{3\sqrt{1+\alpha^2+\beta^2}}.
$$
\end{itemize}
\end{theorem}

\section{Minimal unit vector fields on oscillator group}

Harmonic maps are generalizations of locally minimal isometric immersion. The pulled back metric \eqref{MetricVF} on the submanifold $V(M)\subset (T_1M,g_S)$ manifests that mapping
$V:M\to (T_1M,g_S)$ is isometry if and only if $V$ is a parallel unit vector field. It means that in general one can not expect that a unit vector field which defines a harmonic map is minimal. Surprisingly, there are examples when it happens. For example,  J. C. Gonz\'alez-D\'avila and   L.  Vanhecke \cite{GDVh} proved that a left-invariant unit vector field on a three-dimensional  unimodular Lie group is harmonic if and only if it is minimal; for generalized Heisenbrg group $H(n,1)\  (n\geq 2)$, the set of left-invariant harmonic unit vector fields is given by ${\pm \xi},S\cap \xi^\perp$, where $\xi$ is a distinguished unit vector field and $S\cap \xi^\perp$ is a set of unit fector fields in orthogonal complement of $\xi$,  and they all determine harmonic maps into the unit tangent bundle. The same authors proved that left-invariant vectors of canonical frame $\xi, e_1, \ldots, e_n$ on $H(n,1)$ are minimal as well \cite{GD-Vh}. Recently   Na Xu and Ju Tan \cite{Xu-Tan} find all the left-invariant harmonic unit vector fields on the oscillator
groups. Besides, they determine the associated harmonic maps from the oscillator group into its unit tangent bundle equipped with the associated Sasaki metric. Our main goal is to understand if invariant unit vector fields
on oscillator group which defines a harmonic map can be minimal.

 {The oscillator group} $G_n(\lambda) = G(\lambda_1, \dots , \lambda_n)$ is the connected simply connected
solvable Lie group whose Lie algebra $\mathfrak{g}_n(\lambda)$ is the oscillator algebra
$\mathfrak{g}_n(\lambda) = (\lambda_1, \dots , \lambda_n)$ which is linearly spanned by (2n + 2)-elements
\begin{equation}\label{Frame}
e_1, \dots , e_n; e_{n+1}, \dots , e_{2n}; e_{2n+1}=\xi , e_{2n+2}=\zeta.
\end{equation}
Denote
$$\mathcal{X}=Span(e_1, \dots , e_n), \mathcal{Y}=Span(e_{n+1}, \dots , e_{2n}), \mathcal{P}=Span(\xi), \mathcal{Q}=Span(\zeta). $$
According to \eqref{Frame} ve have the following decomposition:
$$
\mathfrak{g}_n(\lambda)=\mathcal{X}\oplus \mathcal{Y}\oplus \mathcal{P}\oplus \mathcal{Q}.
$$
Let
$$
V=\sum a_i e_i +\sum a_{n+i} e_{n+i} +a_{2n+1} \xi  +a_{2n}\zeta
$$
be a left invariant unit vector fiend on the oscillator group $G_n(\lambda)$. Then
$$
V=V_\mathcal{X}+V_\mathcal{Y}+a_{2n+1} \xi  +a_{2n}\zeta
$$
Introduce the following linear operators and formes similar to ones that are used in an of almost contact geometry:
$$
\vf (V)=\sum( -a_{n+i} e_i +a_i e_{n+i}),    \quad \eta(V)= g(V,\xi)=a_{2n+1},  \quad \theta(V)=g(V,\zeta)=a_{2n} .
$$
As a straightforward consequences of the definitions we get
\begin{equation}\label{Propts}
\begin{array}{l}
 \vf^2=-I+\eta\otimes\xi+\theta\otimes \zeta \quad \vf\xi=0,\quad \vf\zeta=0,\quad \eta\vf=\theta\vf=0, \\[1ex]
\eta(X)=g(X,\xi),\quad  \theta(X)=g(\zeta,X),\\[1ex]
\eta(\xi)=1, \quad \eta(\zeta)=0, \quad \theta(\zeta)=1,\quad  \theta(\xi)=0, \\[1ex]
g(\vf X,\vf Y)=g(X,Y)-\eta(X)\eta(Y)-\theta(X)\theta(Y).
\end{array}
\end{equation}
It easily follows that $\vf$ is skew symmetric
\begin{equation}\label{phi_prop2}
g(\varphi X, Y)=-g(X, \varphi Y), \quad  \varphi^t=-\varphi
\end{equation}
and orthogonal being restricted on $\mathcal{X}\oplus \mathcal{Y}$
\begin{equation}\label{orth_Op}
(\varphi^t\varphi)X =(\varphi\varphi^t)X=X
\end{equation}
for any $X\in \mathcal{X}\oplus \mathcal{Y}$. Relative to the introduced notations,
\begin{equation}\label{VF}
V=V_\mathcal{X}+V_\mathcal{Y} +\eta(V) \xi  +\theta(V)\zeta.
\end{equation}
Remark, also, that the canonical frame \eqref{Frame} is such that
$$e_{n+i}=\vf(e_i), \ \ (i=1,\ldots, n).
$$

\begin{lemma}\label{A}
The Nomizu operator for arbitrary vector field $V$  in case of oscillator group can be given in invariant form as follows:
$$
A_V=\frac12 \big(\eta(V)\vf +\vf(V)\eta+\vf^\flat (V)\xi-2\mathbf{E}_\lambda \vf(V) \theta\big),
$$
where $\vf^\flat=g(\cdot,\vf(V))$ ,  $\mathbf{E}_\lambda=diag(\lambda_1,\ldots,\lambda_n;\lambda_1,\ldots,\lambda_n)$.

\end{lemma}
\begin{proof}
The  non-vanishing Lie brackets of the frame \eqref{Frame} are \cite{Biggs}, \cite{Xu-Tan}:
\begin{equation}\label{Bk}
[e_i, e_{n+j}] = \delta_{ij}\xi , \ [\zeta , e_j] = \lambda_je_{n+j}, \
[\zeta , e_{n+j}] = -\lambda_je_j,\  ( i, j =1\dots, n).
\end{equation}

Using Koszul formula  and  \eqref{Bk}, one can determine the Levi-Civita connection on $G_{n}(\lambda)=$ $G\left(\lambda_{1}, \ldots, \lambda_{n}\right)$ as follows \cite{Xu-Tan}:
\begin{equation}\label{CDcoord}
\begin{array}{llll}
\nabla_{\xi } \xi =0, & \quad \nabla_{\xi } \zeta =0, &\quad \nabla_{\xi } e_{j}=-\frac{1}{2} e_{n+j}, & \quad\nabla_{\xi } e_{n+j}=\frac{1}{2} e_{j}, \\[1ex]
 \nabla_{\zeta } \xi =0, & \quad\nabla_{\zeta } \zeta =0, &\quad \nabla_{\zeta } e_{j}=\lambda_{j} e_{n+j}, &\quad \nabla_{\zeta } e_{n+j}=-\lambda_{j} e_{j}, \\[1ex]
\nabla_{e_{j}} \xi =-\frac{1}{2} e_{n+j}, &\quad \nabla_{e_{j}} \zeta =0, & \quad\nabla_{e_{j}} e_{k}=0, & \quad\nabla_{e_{j}} e_{n+k}=\frac{1}{2} \delta_{j k} \xi , \\[1ex]
 \nabla_{e_{n+j}} \xi =\frac{1}{2} e_{j}, & \quad\nabla_{e_{n+j}} \zeta =0, &\quad \nabla_{e_{n+j}} e_{k}=-\frac{1}{2} \delta_{j k} \xi , &\quad \nabla_{e_{n+j}} e_{n+k}=0 .
\end{array}
\end{equation}
where $\delta_{j k}$ is the  Kronecker symbol and  $1 \leqslant j, k \leqslant n$,

For a left invariant vector field $$V=\sum_{i=1}^{n}\left(a_{i} e_{i}+a_{n+i} e_{n+i}\right)+a_{2 n+1} \xi+a_{2 n+2} \zeta$$ on $G_{n}(\lambda)$
we have \cite{Xu-Tan}
$$
\begin{array}{ll}
\nabla_{e_{j}} V  =\frac{1}{2} a_{n+j} \xi -\frac{1}{2} a_{2 n+1} e_{n+j}, &\quad j=1,2, \ldots, n \\[1ex]
\nabla_{e_{n+j}} V =-\frac{1}{2} a_{j} \xi +\frac{1}{2} a_{2 n+1} e_{j}, & \quad j=1,2, \ldots, n \\[1ex]
\nabla_{\xi } V  =\frac{1}{2} \sum_{i=1}^{n}\left(a_{n+i} e_{i}-a_{i} e_{n+i}\right),& \\[1ex]
\nabla_{\zeta } V  =\sum_{i=1}^{n} \lambda_{i}\left(a_{i} e_{n+i}-a_{n+i} e_{i}\right)& .
\end{array}
$$
If
$
Z=\sum_{j=1}^{n}\left(z_{j} e_{j}+z_{n+j} e_{n+j}\right)+z_{2 n+1} \xi+z_{2 n+2} \zeta
$
is another left  invariant vector field, then
\begin{multline*}
\nabla_Z V=\sum_{j=1}^n(z_j \nabla_{e_{j}} V+z_{n+j} \nabla_{e_{n+j}} V)+z_{2n+1}\nabla_{\xi } V+z_{2n+2} \nabla_{\zeta } V=\\
\sum_{j=1}^n\left(\frac{1}{2}z_j  (a_{n+j} \xi -a_{2 n+1} e_{n+j})-\frac{1}{2}z_{n+j} ( a_{j} \xi - a_{2 n+1} e_{j})\right)+\\
\frac{1}{2}z_{2n+1} \sum_{j=1}^{n}\left(a_{n+j} e_{j}-a_{j} e_{n+j}\right)+z_{2n+2} \sum_{j=1}^{n} \lambda_{j}\left(a_{j} e_{n+j}-a_{n+j} e_{j}\right)=\\
-\frac12 a_{2n+1} \sum_{j=1}^n(-z_{n+j} e_j +z_j e_{n+j}   )-\frac12 z_{2n+1}\sum_{j=1}^{n}\left(-a_{n+j} e_{j}+a_{j} e_{n+j}\right) +\\
z_{2n+2} \sum_{j=1}^{n} \lambda_{j}\left(-a_{n+j} e_{j}+a_{j} e_{n+j}\right)+\frac{1}{2}\sum_{j=1}^n(-z_{n+j} a_{j}+ z_j  a_{n+j}) \xi=\\
-\frac12 \eta(V)\vf(Z) -\frac12\eta(Z)\vf(V)+\theta(Z)\,\mathbf{E}_\lambda \vf(V)+\frac12 g(V,\vf(Z)\xi.
\end{multline*}
So, we have
\begin{multline}\label{CD}
A_VZ=-\nabla_ZV=\\
\frac12 \eta(V)\vf(Z) +\frac12\eta(Z)\vf(V)-\theta(Z)\,\mathbf{E}_\lambda \vf(V)+\frac12 g(\vf(V),Z)\xi
\end{multline}
which completes the proof.
\end{proof}
\begin{corollary} As a Riemannian manifold, the oscillator group $G_n(\lambda)$ splits into metric product $H(n,1)\times I\!\!R$, where $H(n,1)$ is Heisenberg group and $I\!\!R$, and hence, $H(n,1)$ is totally geodesic submanifold in oscillator group $G_n(\lambda)$.
\end{corollary}
\begin{proof}
It is known  that, as a group,  the oscillator group is a semidirect product  $G_\lambda=H(n,1)\rtimes I\!\!R$ of $(2n+1)$-dimensional Heizenberg group $H(n,1)$ and one-dimensional abelian group $I\!\! R$ with Lie algebra $\zeta$.  Plugging $V=\zeta$ into \eqref{CD} and using \eqref{Propts} we get $\nabla_Z\zeta=0$ for all $Z$, and $\nabla_\zeta\zeta=0$. It means that $G_\lambda$ admits a parallel vector field $\zeta$. As a consequence, $G_\lambda$ is a Riemannian product of Heisenberg group $H(n,1)$ and 1-dimensional geodesic $\zeta$-foliation.
\end{proof}
\begin{lemma}\label{A_V}
With respect to canonical frame \eqref{Frame}, matrix of the Nomizu operator for the left invariant vector field $V=V_\mathcal{X}+V_\mathcal{Y} +\eta(V) \xi  +\theta(V)\zeta $ is of the form
$$
 A_V =
\left( \begin{array}{cccc}
0 & -\frac{1}{2}\ \eta(V)I_n & -\frac{1}{2}V_\mathcal{Y} & E_{\lambda }V_\mathcal{Y} \\[1ex]
\frac{1}{2}\ \eta(V)I_n & 0 & \frac{1}{2}V_\mathcal{X} & -E_{\lambda }V_\mathcal{X} \\[1ex]
-\frac{1}{2}V^t_\mathcal{Y} & \frac{1}{2}V^t_\mathcal{X} & 0 & 0 \\[1ex]
0 & 0 & 0 & 0 \end{array}
\right),
$$
where $I_n$ is a  $n\times n$ unit matrix, and $E_{\lambda}=diag(\lambda_1,\ldots,\lambda_n)$.
\end{lemma}
\begin{proof}
Columns of $A_V $ are formed by $A_V e_i, A_V e_{n+i}, A_V\xi$ and $A_V\zeta$. Using \eqref{CD} and \eqref{Propts}, we obtain
$$
\begin{array}{l}
A_V e_i=\frac12 \eta(V)e_{n+i} -\frac12 a_{n+i}\xi, \quad A_V e_{n+i}=-\frac12 \eta(V)e_{i} +\frac12 a_{i}\xi,\\[1ex]
A_V\xi=\frac12 \vf(V)=\frac12 (-V_\mathcal{Y}+V_\mathcal{X}), \quad A_V\zeta=-\mathbf{E}_\lambda \vf(V)=E_\lambda V_\mathcal{Y}-E_\lambda V_\mathcal{X}.
\end{array}
$$
\end{proof}

\begin{lemma}\label{Tensor}
The invariant form of Riemannian curvature tensor of the oscillator  group $G_n(\lambda)$ can be given as follows:
\begin{multline*}
R(X,Y)Z=\frac12 g(\vf (X),Y)\vf(Z)+\frac14 \big(g(\vf(X),Z)\vf(Y)-g(\vf(Y),Z)\vf(X)\big)+\\
\frac14 \big(\eta(Y) X-\eta(X)Y\big)\eta(Z)+\frac14 \big(\eta(X)g(Y,Z)-\eta(Y) g(X,Z)\big)\xi +\\
\frac14\big(\eta(X)\theta(Y)-\eta(Y)\theta(X)\big)\big(\eta(Z)\zeta-\theta(Z)\xi\big).
\end{multline*}
\end{lemma}

\begin{proof} With respect to the frame \eqref{Frame} the components of Riemannian tensor are \cite{Xu-Tan}:
$$
\begin{aligned}
& R\left(e_{i}, e_{n+j}\right) e_{s}=-\frac{1}{2} \delta_{i j} e_{n+s}-\frac{1}{4} \delta_{j s} e_{n+i}, \quad R\left(e_{i}, e_{n+j}\right) e_{n+s}=\frac{1}{2} \delta_{i j} e_{s}+\frac{1}{4} \delta_{i s} e_{j}, \\
& R\left(e_{i}, e_{j}\right) e_{n+s}=\frac{1}{4}\left(\delta_{j s} e_{n+i}-\delta_{i s} e_{n+j}\right), \quad R\left(e_{n+i}, e_{n+j}\right) e_{s}=\frac{1}{4}\left(\delta_{j s} e_{i}-\delta_{i s} e_{j}\right), \\
& R\left(e_{i}, \xi \right) e_{j}=\frac{1}{4} \delta_{i j} \xi , \quad R\left(e_{i}, \xi \right) \xi =-\frac{1}{4} e_{i}, \quad R\left(e_{n+i}, \xi \right) e_{n+j}=\frac{1}{4} \delta_{i j} \xi,\\
&  R\left(e_{n+i}, \xi \right) \xi =-\frac{1}{4} e_{n+i},
\end{aligned}
$$
where $1 \leqslant i, j, s \leqslant n$.  Letting
$$
\begin{array}{l}
X=\sum_{i=1}^n (x_i e_i + x_{n+i} e_{n+i}) +x_{2n+1} \xi  +x_{2n}\zeta, \\[1ex]
Y=\sum_{i=1}^n (y_i e_i + y_{n+i} e_{n+i}) +y_{2n+1} \xi  +y_{2n}\zeta, \\[1ex]
Z=\sum_{i=1}^n (z_i e_i + z_{n+i} e_{n+i}) +z_{2n+1} \xi  +z_{2n}\zeta,
\end{array}
$$
after rather long accurate computations we get what was claimed.

\end{proof}

\begin{theorem}\label{Main}
The oscillator group $G_n(\lambda)$ as Riemannian manifold with the left invariant metric is isometric to the Riemannian product $H(n,1)\times \mathbb{R}$ of $(2n+1)$-dimensional generalized Heisenberg group $H(n,1)$ and abelian subgroup $\mathbb{R}$  with Lie algebras  $\mathfrak{g}_{2n+1}=Span(\xi ,e_1, \dots , e_n; e_{n+1}, \dots , e_{2n})$ and $Span(\zeta )$, respectively. On $G_n(\lambda)$
\begin{itemize}\itemsep=1ex
\item[(a)] $\zeta $ is parallel and hence minimal and totally geodesic;
\item[(b)] \big($\xi ,\vf,g,\eta=g(\cdot,\xi)\big)$ defines the $\alpha$-Sasakian structure on $H(n,1)$ with $\alpha=\frac12$;
\item[(c)] $\xi $ is minimal but not totally geodesic;
\item[(d)] the unit $V\in \xi ^\perp\subset \zeta ^\perp$ is minimal if and only if $\lambda_i^2=\lambda_j^2$ for all $(i,j=1,\ldots,n)$. Moreover
  \begin{itemize}
\item in case $\lambda_1=\ldots=\lambda_n$, each unit $V\in \xi ^\perp\subset \zeta ^\perp$ is  minimal;
\item in case $\lambda_1=\ldots=\lambda_k=-\lambda_{k+1}=\ldots=-\lambda_n$, the field $V=\sum _{i=1}^n(a_i e_i+a_{n+i}e_{n+i})$  is minimal if
$$
{\sum_{i=1}^k(a^2_i+a^2_{n+i})-\sum_{i=k+1}^n(a^2_i+a^2_{n+i})=0.}
$$
\end{itemize}
\end{itemize}
\end{theorem}

\begin{proof}  (a). \ If $V=\zeta$, then Lemma \ref{A_V} implies $A_\zeta=0$ and  \eqref{eqn:6}  implies $\zeta $ is minimal and totally geodesic.\vspace{1em}

(b) Restrict out considerations on totally geodesic $H(n,1)\subset G_n(\lambda)$.  Then $\forall Z$ we have $ \theta(Z)=0 $ and \eqref{Propts} imply
$$
\begin{array}{l}
 \vf^2=-I+\eta\otimes\xi \quad \vf\xi=0,\quad \eta\vf=0, \\[1ex]
\eta(X)=g(X,\xi),\quad \eta(\xi)=1, \\[1ex]
 g(\vf X,\vf Y)=g(X,Y)-\eta(X)\eta(Y).
\end{array}
$$
To prove that $H(n,1)$ is $\alpha$ -Sasakian (cf. \eqref{Basic}, \eqref{BasicMetric} and \eqref{Trans_Sasaki}) we only need to prove that
$$
(\nabla_Z \varphi)W =\alpha(g(Z,W)\xi - \eta(W)Z)  \quad \forall W,Z\in \ker(\theta).
$$
If $V=\xi$, then  Lemma \ref{A} implies
$
A_\xi =\frac12 \vf.
$
In other words,
$$
\vf(Z)=-2\nabla_Z\xi =2A_\xi Z.
$$
By definition,
$$
(\nabla_Z \varphi)W=\nabla_Z(\vf(W))-\vf (\nabla_ZW).
$$
According to \eqref{CD}, for all $W,Z\in \ker(\theta)$
$$
\nabla_ZW=-A_Z W =
-\frac12 \eta(Z)\vf(W) -\frac12 \eta(W)\vf(Z) -\frac12 g(\vf(Z),W)\xi.
$$
Thus, we have
$$
\nabla_Z(\vf(W))=-\frac12 \eta(Z)\vf^2(W) -\frac12 g(\vf(Z),\vf(W))\xi.
$$
$$
\vf(\nabla_ZW)=-\frac12 \eta(Z)\vf^2(W) -\frac12 \eta(W)\vf^2(Z)
$$
So, we have
\begin{multline*}
(\nabla_Z \varphi)W=-\frac12 \eta(W)\vf^2(Z)+\frac12 g(\vf(Z),\vf(W))\xi=\\
-\frac12\big( \eta(W)(-Z+\eta(Z)\xi)+(g(Z,W)-\eta(Z)\eta(W))\xi\big)=\frac12\big( -g(Z,W)\xi+\eta(W)Z \big)
\end{multline*}
Hence, the subgroup $H(n,1)$ carries the $\alpha$-Sasakian structure with $\alpha=\frac12$.\vspace{1em}

(c)  Since the submanifold $H(n,1)$ carries the $\alpha$-Sasakian structure with $\alpha=\frac12$ and is totally geodesic in $G_n(\lambda)$, one can apply Theorems \ref{ReebTG} and \ref{ReebH} to see that $\xi$, being the Reeb vector field on $H(n,1)$, is minimal but not totally geodesic vector field on $H(n,1)$ as well as on $G_n(\lambda)$.\vspace{1em}

(d) If $V\subset\xi^\perp\subset\zeta^\perp$, then $\eta(V)=\theta(V)=0$, and \eqref{CD} implies $\nabla_VV=0$. So, $V$ is \textit{geodesic} vector field on $G_n(\lambda)$.  The matrix of Nomizu operator  takes the form
$$
 A_V =
\left( \begin{array}{cccc}
0 & 0 & -\frac{1}{2}V_\mathcal{Y} & E_{\lambda }V_\mathcal{Y} \\[1ex]
0 & 0 & \frac{1}{2}V_\mathcal{X} & -E_{\lambda }V_\mathcal{X} \\[1ex]
-\frac{1}{2}V^t_\mathcal{Y} & \frac{1}{2}V^t_\mathcal{X} & 0 & 0 \\[1ex]
0 & 0 & 0 & 0 \end{array}
\right)=
\left( \begin{array}{ccc}
 0_{2n} & \frac12 \vf(V) & -\textbf{E}_{\lambda }\vf(V) \\[1ex]
\frac12\vf(V)^t & 0&0 \\[1ex]
0 & 0 & 0 \end{array}
\right),
$$
where $0_{2n}$ is $2n\times 2n$ zero matrix. Then
$$
 A^t_V A_V=
\left( \begin{array}{ccc}
 0_{2n} &  \frac12 \vf(V) &0 \\[1ex]
\frac12 \vf(V)^t & 0&0 \\[1ex]
 -\vf(V)^t\textbf{E}_{\lambda } & 0 & 0 \end{array}
\right)\left( \begin{array}{ccc}
 0_{2n} & \frac12 \vf(V) & -\textbf{E}_{\lambda }\vf(V) \\[1ex]
\frac12\vf(V)^t & 0&0 \\[1ex]
0 & 0 & 0 \end{array}
\right)=
$$
$$
\left( \begin{array}{ccc}
\frac14\varphi (V)\varphi(V)^ t & 0 & 0 \\[1ex]
0 & \frac14\varphi(V)^ t \varphi(V)  & -\frac12\varphi(V)^t \mathbf{E}_\lambda \varphi(V) \\[2ex]
0 & -\frac12\varphi(V)^ t\mathbf{E}_\lambda \varphi (V) &  \varphi(V)^t\mathbf{E}^2_\lambda\varphi(V)\end{array}
\right).
$$
Observe, that
$$\varphi(V)^ t \varphi(V)=|\varphi(V)^2|=|V|^2=1,$$
$$ \varphi(V)^ t\mathbf{E}_\lambda \varphi (V)=\sum_{j=1}^n \lambda_j(a_{n+j}^2+a_j^2)=V^ t\mathbf{E}_\lambda, V$$
$$
\varphi(V)^ t\mathbf{E}^2_\lambda \varphi (V)=\sum_{j=1}^n \lambda^2_j(a_{n+j}^2+a_j^2)=V^ t\mathbf{E}^2_\lambda V.
$$
So finally,
$$
 A^t_V A_V=
\left( \begin{array}{ccc}
\frac14\varphi (V)\varphi(V)^ t & 0 & 0 \\[1ex]
0 & \frac14  & -\frac12 V^t \mathbf{E}_\lambda  V \\[2ex]
0 & -\frac12 V^ t\mathbf{E}_\lambda  V &   V^t\mathbf{E}^2_\lambda V
\end{array}
\right).
$$
The matrix $A_V^t A_V$ is a block diagonal. It is easy to see that
$$
rank \left(\varphi (V){\varphi }^ t(V)\right)=1, \quad  trace\left(\varphi (V)\varphi (V)^t\right) =\frac14 |V|^2=\frac14
$$
Denote by $\sigma_1,\ldots,\sigma_{2n+2}$ a singular values and by bold $\mathbf{e}_1,\ldots, \mathbf{e}_{2n+2}$ a corresponding singular  orthonormal frame for $A_V$. Observe, that
  ${\sigma }^2=0$ is the eigenvalue of multiplicity  $2n-1$, while  ${\sigma }^2=\frac{1}{4}$
is the eigenvalue of multiplicity 1 with $\varphi (V)$ as the eigenvector.
The rest of singular values are the roots of
\begin{equation}\label{Poly}
\mu ^2-\left(\frac14+V^ t E^2_\lambda V\right)\mu  +\frac14 \left(V^ t E^2_\lambda V -\left(V^t E_\lambda V\right)^2\right)=0.
\end{equation}
In coordinate form the equation above can be expressed as
$$
\mu ^2-\left(\frac14+\sum_{i=1}^n \lambda_i^2(a_i^2+a_{n+i}^2) \right)\mu  +\frac18 \sum_{i,j=1}^n(\lambda_i-\lambda_j)^2(a_i^2+a_{n+i}^2)(a_j^2+a_{n+j}^2)=0.
$$
The corresponding singular vectors are of the form
$$
\mathbf{e}_{2n+1}=\frac{1}{\sqrt{{\alpha }^2+{\beta }^2}}\left(\alpha \xi +\beta \ \zeta \right),\quad \mathbf{e}_{2n+2}=\frac{1}{\sqrt{{\alpha }^2+{\beta }^2}}\left(-\beta \xi +\alpha \ \zeta \right)
$$
As a result, the singular frame eigenpairs are
$$
(\sigma_1=\ldots=\sigma_{2n-1}=0;\ \mathbf{e}_1=V,\ \mathbf{e}_2,\dots ,\mathbf{e}_{2n-1}),
$$
where $ (\mathbf{e}_2,\dots ,\mathbf{e}_{2n-1})$ is arbitrary orthonormal frame in $ V^\perp \subset \vf(V)^\perp$,
$$
\left(\sigma_{2n}=\frac{1}{4}\ ,\mathbf{e}_{2n}=\varphi (V)\right),\quad \left(\mu_1,\ \mathbf{e}_{2n+1}\right),\ (\mu_2,\ \mathbf{e}_{2n+2}).
$$
From definition of singular frame it follows that $A_V\mathbf{e}_i=0, \ (i=1,\dots 2n-1)$.  Since $\mathbf{e}_i\in \mathcal{X}\oplus \mathcal{Y}\subset\ker(\eta)\subset \ker(\theta)$ and $\vf$ is skewsymmetric on $\mathcal{X}\oplus \mathcal{Y}$, the  \eqref{CD} implies $\nabla_{\e_i}\e_i=0$. Therefore,
$$
  (\nabla _{\mathbf{e}_i}A_V)\mathbf{e}_i=\nabla_{\e_i} (A_V\e_i)-A_V(\nabla_{\e_i}e_i)=0,\quad R(V,A_V\mathbf{e}_i)\mathbf{e}_i=0\quad (i=1,\dots ,\ 2n-1).
$$
Now, compute
$$
(\nabla_{\mathbf{e}_{2n}}A_V )\mathbf{e}_{2n}=\nabla_{\mathbf{e}_{2n}}(A_V \mathbf{e}_{2n}) -A_V(\nabla_{\mathbf{e}_{2n}}\e_{2n}).
$$
Since $\e_{2n}=\vf(V)$ the \eqref{CD} implies
$$
A_V \mathbf{e}_{2n}=\frac12 g(\vf( V),\vf(V))\xi=\frac12\xi, \quad \nabla_{\mathbf{e}_{2n}}(A_V \mathbf{e}_{2n})=-\frac12\nabla_{\vf(V)}\xi=-\frac14 \vf^2 V=\frac14 V .
$$
Compute
$$
\quad R(V,A_V\mathbf{e}_{2n}) \mathbf{e}_{2n}=\frac12R(V,\xi)\vf(V).
$$
Using Lemma \ref{Tensor} restricted to $\mathcal{X}\oplus \mathcal{Y}$ we get
$$
R(V,\xi)\vf(V)=\frac12 g(\vf (V),\xi )\vf^2(V)+\frac14 \big(g(\vf(V),\vf(V))\vf(\xi)-g(\vf(\xi ),\vf(V))\vf(V)\big)=0.
$$

For  $\mathbf{e}_{2n+1}=\frac{1}{\sqrt{{\alpha }^2+{\beta }^2}}(\alpha \xi +\beta \ \zeta )$  we observe, first, that \eqref{CD} implies
$$
A_V\xi =\frac12  \vf(V), \quad A_V\zeta =-\frac12 \mathbf{E}_\lambda\vf(V)
$$
and hence
$$
A_V\mathbf{e}_{2n+1}=\frac{\left(\alpha I_{2n}-2\beta \mathbf{E}_\lambda \right)}{2\sqrt{\alpha ^2+\beta ^2}}\ \varphi (V),
\qquad \sigma^2_{2n+1}=\left|\frac{\left(\alpha I_{2n}-2\beta \mathbf{E}_\lambda \right)}{2\sqrt{\alpha ^2+\beta ^2}}\ \varphi (V)\right|^2,
$$
where $I_{2n}$ is $2n\times 2n$ unit matrix.
Further on,  from \eqref{CDcoord} it follows that $\nabla_\xi\zeta=\nabla_\zeta\xi=\nabla_\xi\xi=\nabla_\zeta\zeta=0$ and hence
\begin{multline*}
(\nabla _{ \mathbf{e}_{2n+1}} A_ V )\mathbf{e}_{ 2n+1}=\nabla_{\e_{2n+1}}(A_V \e_{2n+1})=\frac{\left(\alpha I_{2n}-2\beta \mathbf{E}_\lambda \right)}{2\sqrt{\alpha ^2+\beta ^2}}\nabla _{ \mathbf{e}_{2n+1}}\varphi (V)=\\
-\frac{\left(\alpha I_{2n}-2\beta \mathbf{E}_\lambda \right)}{2\sqrt{\alpha ^2+\beta ^2}}A_{\vf(V)} \e_{2n+1}
\end{multline*}
Applying again \eqref{CD} we calculate
$$
A_{\vf(V)} \e_{2n+1}=\frac{1}{\sqrt{\alpha^2+\beta^2}}(\alpha A_{\vf(V)}\xi +\beta A_{\vf(V)}\zeta)=-\frac{1}{2\sqrt{\alpha^2+\beta^2}}(\alpha I_{2n}- 2\beta \mathbf{E}_\lambda) V
$$
As a result,
$$
(\nabla _{ \mathbf{e}_{2n+1}} A_ V )\mathbf{e}_{ 2n+1}=\frac{{\left(\alpha I_{2n}-2\beta  \mathbf{E}_\lambda\right)}^2}{4({\alpha }^2+{\beta }^2)}V,
$$
Since $V, A_V\mathbf{e}_{2n+1}\in \mathcal{X}\oplus \mathcal{Y}$ and $\e_{2n+1}$ is a linear combination of $\xi$ and $\zeta$, from Lemma \ref{Tensor} we conclude that
$$
 R\left(V,A_V\mathbf{e}_{2n+1}\right)\mathbf{e}_{2n+1}=0.
$$
In a similar way, for  $\mathbf{e}_{2n+2}=\frac{1}{\sqrt{\alpha ^2+\beta ^2}}(-\beta \xi +\alpha \zeta )$  we get
$$
 A_ V \mathbf{e}_{ 2n+2}=-\frac{\left(\beta I_{2n}{{\mathbf +}2\alpha \mathbf{ E}}_{\lambda }\right)}{2\sqrt{{\alpha }^2+{\beta }^2}}\ \varphi (V), \quad\quad
\sigma ^2_{2n+2}=\left|\frac{\left(\beta I_{2n}+2\alpha \mathbf{ E}_\lambda \right)}{2\sqrt{{\alpha }^2+{\beta }^2}}\ \varphi (V)\right|^2.
$$
$$
(\nabla_ {\mathbf{e}_{2n+2}} A_ V)\mathbf{e}_{2n+2}=
\frac{\left(\beta I_{2n}+2\alpha \mathbf{ E}_{\lambda }\right)^2V}{4(\alpha^2+\beta ^2)}, \quad R\left(V,A_Ve_{2n+2}\right)e_{2n+2}=0
$$
Summing up,
$$
\check{H}_V=\frac{\frac{1}{4}V}{1+\frac{1}{4}}+\frac{1}{1+{\sigma }^2_{2n+1}}\frac{{\left(\alpha I_{2n} -2\beta  \mathbf{E}_\lambda\right)}^2V}{4({\alpha }^2+{\beta }^2)}+\frac{1}{1+{\sigma }^2_{2n+2}}\frac{{\left(\beta I_{2n}+2\alpha \mathbf{E}_\lambda\right)}^2V}{4(\alpha ^2+\beta ^2)}
$$
 The vector field is minimal iff ${\check{H}}_V\parallel V$. It means that the following proportion relation must be fulfilled:
$$
(\alpha I_{2n} -2\beta \mathbf{E}_\lambda)^2\quad\sim\quad I_{2n} ,\quad (\beta I_{2n}+2\alpha \mathbf{E}_\lambda )^2\quad\sim\quad I_{2n}.
$$
It follows
\begin{equation}\label{Eq}
\alpha^2 I_{2n} -\rm 4\alpha \beta \mathbf{ E}_\lambda +4\beta^{2}\mathbf{ E}^{2}_\lambda\ \sim \ I_{2n}, \quad
\beta^{2}I_{2n} +4\alpha \beta \mathbf{ E}_\lambda +4\alpha^2\mathbf{ E}^2_\lambda \ \sim\ I_{2n}.
\end{equation}
Adding, we get
$$
(\alpha^2 +\beta ^2)I_{2n} +4(\alpha ^{2} +\beta^2)\mathbf{E}^2_\lambda\quad\sim\quad I_{2n} \quad \Rightarrow \quad  \mathbf{E}^2_\lambda \quad\sim\quad I_{2n}.
$$
In this case, from \eqref{Eq} we conclude
$$
4\alpha \beta \mathbf{E}_\lambda \quad \sim\quad I_{2n}.
$$
In case $\alpha  =0,\beta\ne0$  we have only  ${E}^2_\lambda\sim I_{2n}$ which implies
$$
\lambda^2_1=\ldots \lambda^2_n=\lambda^2, \quad \sigma ^2_{2n +1}=\lambda^2, \quad \sigma ^2_{2n +2}=\frac14.
$$
The Vieta's theorem for \eqref{Poly} implies that  it  is possible iff $V^t\mathbf{E}_\lambda V=0$.
Then
$$
\check{H}_V=\frac{\frac{1}{4}V}{1+\frac{1}{4}}+\frac{\lambda^2 V}{1+\lambda^2}+\frac{\frac{1}{4}V}{1+\frac{1}{4}}
$$
and hence,  {$V$ is minimal} in assumption  that $V^t\mathbf{E}_\lambda V=0$.

If $\beta=0,\alpha\ne0$, then
$$
\sigma ^2_{2n +1}=\frac14, \quad \sigma ^2_{2n +2}=\lambda^2,
$$
and again the Vieta's theorem for \eqref{Poly} implies that  it  is possible iff $V^t\mathbf{E}_\lambda V=0$. Then
$$
\check{H}_V=\frac{\frac{1}{4}V}{1+\frac{1}{4}}+\frac{\frac{1}{4}V}{1+\frac{1}{4}}+\frac{\lambda^2 V}{1+\lambda^2}
$$
and hence,  {$V$ is minimal}.
The condition ${V^t\mathbf{E}_\lambda V=0}$ make sense if $\mathbf{E}_\lambda \not\sim I_{2n}$  which means that among the structure constants there are ones of differen sign. Suppose $\lambda_1=\ldots=\lambda_k=\lambda>0$ while
$\lambda_{k+1}=\ldots=\lambda_{n}=-\lambda<0$. Then the field $V=\sum _{i=1}^n(a_i e_i+a_{n+i}e_{n+i})$  is minimal if
$$
{V^t\mathbf{E}_\lambda V=0}\quad \sim \quad {\sum_{i=1}^k(a^2_i+a^2_{n+i})-\sum_{i=k+1}^n(a^2_i+a^2_{n+i})=0}
$$
If $\alpha \beta \ne  0$, then
$$
\mathbf{E}_\lambda\quad  \sim \quad I_{2n}
$$
and hence $\lambda _i=\lambda $. In this case $\sigma_{2n +1}=\frac14 +\lambda^2\quad \sigma_{2n +2} =0$, and
$
\alpha =\lambda \quad \beta =-2\lambda^2 ,\quad \alpha^2 +\beta^2=\lambda^{2}(1+4\lambda^{2}).
$
Therefore,
$$
\check{H}_V=\frac{\frac14 V}{ 1+\frac14 } +\frac{\left( 1+\frac{\lambda ^2}{4}\right)V}{ 1+1+\frac{\lambda^2}{4}}=\left(\frac{\sigma^2_1}{1+\sigma^2_1} +\frac{\sigma ^2_{2n +1}}{1+\sigma^{2}_{2n +1}}\right)V
$$
and hence, V is \textbf{minimal} without any restrictions.
The proof is complete.
\end{proof}

Now we can compare our result with the result of Xu and Tan:

\begin{theorem}\cite{Xu-Tan}
The set of leftinvariant unit vector fields on the oscillator group $G_{n}(\lambda)=G\left(\lambda_{1}, \ldots, \lambda_{n}\right)$, such that the corresponding maps into the unit tangent bundles are harmonic, is given by
$$
\{ \pm \xi\} \cup\{ \pm \zeta\} \cup\left(\mathcal{S} \cap\left\{\sum_{j=1}^{n}\left(a_{j} e_{j}+a_{n+j} e_{n+j}\right)\right\}\right)
$$
where for $\lambda_{i}^{2} \neq \lambda_{j}^{2},\left(a_{i}^{2}+a_{n+i}^{2}\right)\left(a_{j}^{2}+a_{n+j}^{2}\right)=0$.
\end{theorem}

\begin{corollary}
 On the oscillator group with structure constants $\lambda_i=\lambda_j \ (\forall i,j=1,\ldots,n)$ {each left invariant unit vector field that defines a \textbf{{harmonic map}}} into the unit tangent bundle with  Sasaki metric {\textbf{is minimal}}.
\end{corollary}
J. C. Gonz\'alez-D\'avila and L. Vanhecke \cite{GD-Vh} considered minimality of leftinvariant unit vector fields on generalized Heisenberg group and proved that the $e_1,\ldots, e_n; e_{n+1},\ldots, e_{2n}, \xi$ of the frame \eqref{Frame} are minimal. The Theorem \ref{Main} presents more general result.


\begin{thebibliography}{10}
\bibitem{Biggs}  Biggs, R., Remsing, C.:  Some remarks on the oscillator group. Diff. Geom. and its Appl. 35, 199 -- 209, (2014). \url{doi.org/10.1016/j.difgeo.2014.03.003}.
\bibitem{GZ} Gluck, H., Ziller, W.: On the volume of a unit vector field on the three sphere. Com. Math Helv. 61, 177 -- 192 (1986). \url{eudml.org/doc/140047}.
\bibitem{GM-LF} Gil-Medrano, O., Llinares-Fuster, E.: Minimal unit vector fields. Toh\^oku Math. J. 54(1), 71 -- 84  (2002).  \url{doi: 10.2748/tmj/1113247180}.
\bibitem{GD-Vh}   Gonz\'alez-D\'avila, J. C.,  Vanhecke, L.: Examples of minimal unit vector fields. Ann. Global Anal. Geom. 18, 385 -- 404 (2000). \url{doi.org/10.1023/A:1006788819180}
\bibitem{GDVh}  Gonz\'alez-D\'avila, J. C.,  Vanhecke, L.:
                Invariant harmonic unit vector fields on Lie groups. Bollettino dell'Unione Matematica Italiana, Serie 8, Vol. 5-B, No.2, 377 -- 403 (2002). \url{www.bdim.eu/item?id=BUMI_2002_8_5B_2_377_0}
\bibitem{Ob}   Obi\~{n}a, J.A.: New Classes of almost Contact metric structures, Publ. Math. Debrecen 32: 187 -- 193 (1985). \url{publi.math.unideb.hu/load_doc.php?p=3718&t=pap}
\bibitem{Marrero}    Marrero, J. C.: The local structure of trans-Sasakian manifolds. Ann. di Mat. Pura ed Appl. 162(4), 77 -- 86 (1992). \url{doi: 1007/BF01760000}
\bibitem{Ym-2002}  Yampolsky, A.: On the mean curvature of a unit vector field, Publ. Math. Debrecen  60, 2/3, 131 -- 155  (2002). \url{DOI: 10.5486/PMD.2002.2562}
\bibitem{Ym-2003}  Yampolsky, A.:  A totally geodesic property of Hopf vector fields, Acta Math. Hung. 101, 93 -- 112 (2003). \url{doi.org/10.1023/B:AMHU.0000003895.90745.38}
\bibitem{Ym-2004} Yampolsky, A.: Full description of totally geodesic unit vector fields on Riemannian 2-manifold. Mat. Fiz. Anal. Geom., 11:3, 355 -- 365 (2004). \url{www.mathnet.ru/eng/jmag213}
\bibitem{Ym-2005} Yampolsky, A.: On special types of minimal and totally geodesic unit vector fields. Proc. 7-th Intl Conf. in Geometry, Integrability and Quantization June 2 - 10,  Varna, Bulgaria, 292 -- 306 (2005). \url{www.emis.de/proceedings/Varna/vol7/P07Yampolsky.pdf}
\bibitem{Ym-2007} Yampolsky, A.: Invariant totally geodesic unit vector fields on three-dimensional Lie groups. J. Math. Phys., Anal., Geom. 3 (2), 253 -- 276 (2007). \url{jmag.ilt.kharkiv.ua/index.php/jmag/article/view/jm03-0253e/521}
\bibitem{Ym-2022} Yampolsky, A.: On properties of the Reeb vector field of $(\alpha, \beta)$ trans-Sasakian structure, Turkish J.  Math, V. 46: No. 6, Art. 19 (2022). \url{doi.org/10.55730/  1300-0098.3271}
\bibitem{Xu-Tan} Xu, N., Tan, J.: Invariant harmonic unit vector fields on the oscillator groups, Czech. Math. J., 69,  4,  907 -- 924 (2019). \url{doi.org/10.21136/CMJ.2019.0538-17}





\end{thebibliography}
\end{document}